\documentclass[11 pt]{article}
\usepackage{graphicx} % Required for inserting images
\usepackage[utf8]{inputenc}
\usepackage[left=1in,top=1in,right=1in,bottom=1in]{geometry}
\usepackage{setspace}
\usepackage{amssymb, amsmath, amsthm, graphicx,mathrsfs}
\usepackage{caption,cite}
\usepackage{mathtools,color, xcolor}
\usepackage{tikz}
\usetikzlibrary{calc}

 \usepackage{enumitem}

\makeatletter
\newcommand\thankssymb[1]{\textsuperscript{\@fnsymbol{#1}}}
\makeatother

\makeatletter\let\@@citation@@=\citation\renewcommand{\citation}[1]{\@@citation@@{#1}\@for\@tempa:=#1\do{\@ifundefined{cit@\@tempa}{\global\@namedef{cit@\@tempa}{}}{}}}\makeatother
\usepackage[linktocpage=true]{hyperref}
\usepackage[capitalize]{cleveref}
\makeatletter\def\@lbibitem[#1]#2#3\par{\@ifundefined{cit@#2}{}{\@skiphyperreftrue\H@item[\ifx\Hy@raisedlink\@empty\hyper@anchorstart{cite.#2\@extra@b@citeb}\@BIBLABEL{#1}\hyper@anchorend\else\Hy@raisedlink{\hyper@anchorstart{cite.#2\@extra@b@citeb}\hyper@anchorend}\@BIBLABEL{#1}\fi\hfill]\@skiphyperreffalse}\if@filesw\begingroup\let\protect\noexpand\immediate\write\@auxout{\string\bibcite{#2}{#1}}\endgroup\fi\ignorespaces\@ifundefined{cit@#2}{}{#3}} \def\@bibitem#1#2\par{\@ifundefined{cit@#1}{}{\@skiphyperreftrue\H@item\@skiphyperreffalse\Hy@raisedlink{\hyper@anchorstart{cite.#1\@extra@b@citeb}\relax\hyper@anchorend}}\if@filesw\begingroup\let\protect\noexpand\immediate\write\@auxout{\string\bibcite{#1}{\the\value{\@listctr}}}\endgroup\fi\ignorespaces\@ifundefined{cit@#1}{}{#2}}\makeatother
\newtheorem{thm}{Theorem}[section]

\newtheorem{defn}[thm]{Definition}

\newtheorem{lem}[thm]{Lemma}

\newtheorem{prob}{Problem}

\usepackage{stmaryrd,upgreek,comment}

\newcommand{\eps}{\upvarepsilon}

\newcommand{\cS}{\mathcal{S}}

\renewcommand{\rm}[1]{{\mathrm{#1}}}

\DeclareMathOperator{\bicone}{bicone}
\newcommand{\bcn}{\rm{bc}}

\title{Hypergraph Ramsey numbers with quasipolynomial growth rate}

\author{Xiaoyu He \thanks{School of Mathematics, Georgia Institute of Technology, Atlanta, GA 30332. Email: xhe399@gatech.edu.}
\and Jiaxi Nie\thanks{School of Mathematics, Georgia Institute of Technology, Atlanta, GA 30332. Email: jnie47@gatech.edu.} \and Logan Post\thanks{School of Mathematics, Georgia Institute of Technology, Atlanta, GA 30332. Email: lpost3@gatech.edu.} \and Jacques Verstraëte\thanks{Department of Mathematics, University of California, San Diego, La Jolla, CA 92093. Email: jacques@ucsd.edu}}
\date{October 2025}

\begin{document}

\maketitle

\begin{abstract}
For a 3-uniform hypergraph (3-graph) $F$, let $r(F,n)$ be the smallest $N$ such that any $N$-vertex $F$-free 3-graph has an independent set of size $n$. We construct a $3$-graph $H_2$ with six vertices and five edges such that $r(H_2,n)=n^{\Theta(\log n)}$, and a more general family of $3$-graphs $F$ for which $r(F,n)=n^{\log^{\Theta(1)}(n)}$. These are the first examples of such Ramsey number known to be neither polynomial nor exponential. 
\end{abstract}

\section{Introduction}

Given $k$-uniform hypergraphs $F$ and $G$ (henceforth $k$-graphs), the Ramsey number $r(F,G)$ is the smallest integer $N$ such that every red-blue edge coloring of the complete $k$-graph $K^{(k)}_N$ contains either a red $F$ or a blue $G$. Let $r(F,n)=r(F,K^{(k)}_n)$. In other words, $r(F,n)$ is the smallest integer $N$ such that every $N$-vertex $k$-graph contains either a subgraph isomorphic to $F$ or an independent set of size $n$. For graphs (the case \(k = 2\)), it is known that \(r(F, n)\) always grows polynomially in \(n\), and one of the central problems of Ramsey theory is to determine the precise polynomial order (see e.g. \cite{AKS1980, Kim1995,SHE1983,bohman2010early,hefty2025improving} for $F=K_3$, \cite{mattheus2023asymptotics} for $F = K_4$, and \cite{caro2000asymptotic} for bipartite $F$).

For any $3$-graph $F$, Erd\H{o}s and Rado~\cite{erdos1952combinatorial} proved that $r(F,n)$ is at most exponential in $n$\footnote{In this paper we say a function is exponential in $n$ if it is of the form $2^{n^{\Theta(1)}}$.}. Furthermore, if $F$ is \emph{iterated tripartite} (see~\cite{conlon2010hypergraph,fox2021independent}), Erd\H{o}s and Hajnal~\cite{erdos1972ramsey} proved that $r(F,n)$ grows only polynomially in $n$. It was conjectured in~\cite{conlon2024when} that this characterization is exact: namely, that $r(F,n)$ is polynomial in $n$ if and only if $F$ is iterated tripartite. Furthermore, in all known cases where $r(F,n)$ grows superpolynomially, the growth rate is exponential.

For example, it was proven by Conlon, Fox, Gunby, He, Mubayi, Suk, Verstra\"ete and Yu~\cite{conlon2024when} that for \textit{tightly connected} 3-graphs, if $F$ is tightly connected and not tripartite, then $r(F,n)\geq 2^{\Omega(n^{2/3})}$.
Erd\H{o}s and Hajnal also introduced a variant of the off-diagonal Ramsey number with the same polynomial-to-exponential transition behavior. Let $r_3(s,t;n)$ be the smallest integer $N$ such that every $N$-vertex $3$-graph contains either a set of $s$ vertices spanning at least $t$ edges or an independent set of size $n$. They conjectured that for any fixed $s,t,k$, $r_k(s,t;n)$ grows either polynomially or exponentially in $n$. This conjecture was recently confirmed by Ascoli, He, and Yu~\cite{ascoli2025polynomial}, building on the work of Conlon, Fox, Sudakov~\cite{conlon2010hypergraph} and Mubayi and Razborov~\cite{MubayiRazborov2021}. 

It is thus tempting to conjecture that for every fixed $F$, the function $r(F,n)$ is either polynomial or exponential in $n$. In this paper we refute this intuition by constructing a family of $3$-graphs $F$ for which $r(F,n)=n^{\log^{\Theta(1)}(n)}$.

 Let $H_1=K_3^{(3)}$ be a single edge and let $H_2$ be the 3-graph with vertex set $\{a,b,c,x,y,z\}$ and edge set $\{abc,xyz,ayz,bxz,cyz\}$; see Figure~\ref{fig:H2(simple)} for an illustration. For $t\ge 1$, $H_{t+1}$ is a supergraph of $H_t$, whose definition is deferred to Section~\ref{sec:2}.

 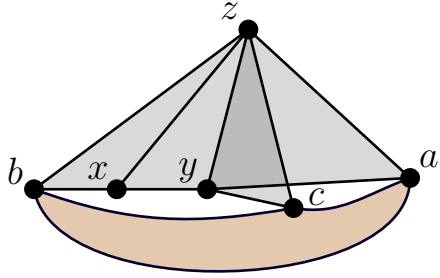
\begin{figure}
     \centering
     \begin{tikzpicture}[scale=0.5,line cap=round,line join=round]

% =========================
% Vertices
% =========================
\coordinate (v1) at (0,0);
\coordinate (v2) at (2.2,0);
\coordinate (v3) at (4.6,0);
\coordinate (v4) at (6.9,-0.5);
\coordinate (v5) at (10,0.3);
\coordinate (v0) at (5.7,4.25);

% =========================
% Hyperedges of the 3-graph
%   012, 023, 034, 045, 056  (grey filled triangles)
%   136                      (blue filled "boat" region)
% =========================

% --- grey triangles for 0ij with consecutive base pairs ---
\foreach \a/\b in {v1/v2,v2/v3,v3/v4,v3/v5}{
  \path[fill=black!60,opacity=0.25] (v0)--(\a)--(\b)--cycle;
}

% --- brown hyperedge 136 as a filled curved region through vertices 1,3,6 ---
\path[fill=brown!60,opacity=0.70]
  (v1)
    to[out=-20,in=-170,looseness=1] (v4)
    to[out=-10,in=-160,looseness=1] (v5)
    to[out=-90,in=-80,looseness=0.8] (v1)
  -- cycle;

% brown boundary (thick)
\draw[line width=1pt,blue!20!black]
  (v1) to[out=-20,in=-170,looseness=1] (v4)
       to[out=-10,in=-160,looseness=1] (v5);
\draw[line width=1pt,blue!20!black]
  (v5) to[out=-90,in=-80,looseness=0.8] (v1);

% =========================
% Black framework (for the illustration)
% =========================
% base polyline
\draw[line width=1.0pt] (v1)--(v2)--(v3)--(v4);
\draw[line width=1.0pt] (v3)--(v5);

% rays from 0 to all base vertices (so the grey triangles are visible)
\foreach \p in {v1,v2,v3,v4,v5}{
  \draw[line width=1.0pt] (v0)--(\p);
}

% =========================
% Vertex markers
% =========================
\foreach \p in {v1,v2,v3,v4,v5,v0}{
  \fill (\p) circle[radius=7.5pt];
}

% =========================
% Labels
% =========================
\node[font=\Large] at ($(v0)+(-0.5,0.5)$) {$z$};
\node[font=\Large] at ($(v1)+(-0.5,0.5)$) {$b$};
\node[font=\Large] at ($(v2)+(-0.5,0.5)$) {$x$};
\node[font=\Large] at ($(v3)+(-0.5,0.5)$) {$y$};
\node[font=\Large] at ($(v4)+(0.6,0.3)$) {$c$};
\node[font=\Large] at ($(v5)+(0.5,0.5)$) {$a$};

\end{tikzpicture}
     \caption{An illustration of $H_2$.}
     \label{fig:H2(simple)}
 \end{figure}

\begin{thm}\label{thm:iteratedbicone} For $t\ge 2$, 
$$
r(H_t,n)\leq n^{O(\log^{2t-3}(n))}.
$$
\end{thm}

The lower bound follows from a recent construction of \cite{conlon2024when}. We say that a $3$-graph $H$ is \emph{tightly connected} if, for any two edges $e$ and $f$ in $H$, there exists a sequence of edges 
\[
e = e_0, e_1, \dots, e_t = f
\] 
such that each consecutive pair $e_{i-1}$ and $e_i$ shares exactly two vertices for all $i = 1, \dots, t$. A \emph{tight component} of $H$ is a maximal tightly connected subset of $E(H)$. 
\begin{thm}[Theorem 1.3 in \cite{conlon2024when}]\label{thm:lower_bound}
If $F$ is a 3-graph with at most two tight components and not iterated tripartite, then
$$
r(F,n)=n^{\Omega(\log n)}.
$$
\end{thm}

We remark that the lower bound construction in \Cref{thm:lower_bound} has a key structural property: roughly speaking, for every sufficiently large subset $V'$ and every vertex $v \in V'$, one can find in the link of $v$ a large bipartite graph with parts $X,Y\subseteq V'$ of size $O(|V'|/n)$ such that no edge goes across $X$ and $Y$, that is, no edge intersects one of them in two vertices and the other in one vertex. Our proof of \cref{thm:iteratedbicone} exactly matches this structure by finding ``well-connected" bipartite graphs of exactly this size in the links of certain vertices. 

Note that $H_2$ has two tight components, one of which is the single edge $abc$, and $H_2$ is not iterated tripartite by inspection, so by 
Theorem~\ref{thm:iteratedbicone} and Theorem~\ref{thm:lower_bound} we have 
$$
r(H_t,n)=n^{\log^{\Theta(1)}(n)} 
% \lp{maybe $2^{\log...}$?}
$$
for every $t\ge 2$ and, in particular,
\begin{equation}\label{eq:quasi}
r(H_2,n)=n^{\Theta(\log n)}.
\end{equation}

% We remark that the lower bound coloring of \Cref{thm:lower_bound} has a key structural property: roughly speaking, the ``shadow" of this coloring on $N=n^{\Theta(\log n)}$ vertices can be partitioned into large monochromatic bicliques of size $N/\textnormal{poly}(n)$. Our proof of \cref{thm:iteratedbicone} exactly matches this structure by finding ``well-connected" bipartite graphs of exactly this size in the links of certain vertices. \xh{Check this wording}

Our results indeed provide a family of hypergraphs satisfying (\ref{eq:quasi}). One simple example is the \emph{sailboat} graph shown in Figure~\ref{fig:sailboat}; see Section 3 (Concluding Remarks) for further discussion.

 \begin{figure}
     \centering
     \begin{tikzpicture}[scale=0.5,line cap=round,line join=round]

% =========================
% Vertices
% =========================
\coordinate (v1) at (0,0);
\coordinate (v2) at (2.2,0);
\coordinate (v3) at (4.6,0);
\coordinate (v4) at (6.9,0);
\coordinate (v5) at (9.1,0);
\coordinate (v6) at (11.4,0);
\coordinate (v0) at (5.7,4.25);

% =========================
% Hyperedges of the 3-graph
%   012, 023, 034, 045, 056  (grey filled triangles)
%   136                      (blue filled "boat" region)
% =========================

% --- grey triangles for 0ij with consecutive base pairs ---
\foreach \a/\b in {v1/v2,v2/v3,v3/v4,v4/v5,v5/v6}{
  \path[fill=black!60,opacity=0.25] (v0)--(\a)--(\b)--cycle;
}

% --- brown hyperedge 136 as a filled curved region through vertices 1,3,6 ---
\path[fill=brown!60,opacity=0.70]
  (v1)
    to[out=-20,in=-160,looseness=1] (v3)
    to[out=-15,in=-165,looseness=1] (v6)
    to[out=-120,in=-80,looseness=0.7] (v1)
  -- cycle;

% brown boundary (thick)
\draw[line width=1pt,blue!20!black]
  (v1) to[out=-20,in=-160,looseness=1] (v3)
       to[out=-15,in=-165,looseness=1] (v6);
\draw[line width=1pt,blue!20!black]
  (v6) to[out=-120,in=-80,looseness=0.7] (v1);

% =========================
% Black framework (for the illustration)
% =========================
% base polyline
\draw[line width=1.0pt] (v1)--(v2)--(v3)--(v4)--(v5)--(v6);

% rays from 0 to all base vertices (so the grey triangles are visible)
\foreach \p in {v1,v2,v3,v4,v5,v6}{
  \draw[line width=1.0pt] (v0)--(\p);
}

% =========================
% Vertex markers
% =========================
\foreach \p in {v1,v2,v3,v4,v5,v6,v0}{
  \fill (\p) circle[radius=7.5pt];
}

% =========================
% Labels
% =========================
\node[font=\Large] at ($(v0)+(-0.5,0.5)$) {$a$};
\node[font=\Large] at ($(v1)+(-0.5,0.5)$) {$b$};
\node[font=\Large] at ($(v2)+(-0.5,0.5)$) {$c$};
\node[font=\Large] at ($(v3)+(-0.5,0.5)$) {$d$};
\node[font=\Large] at ($(v4)+(0.5,0.5)$) {$e$};
\node[font=\Large] at ($(v5)+(0.5,0.5)$) {$f$};
\node[font=\Large] at ($(v6)+(0.5,0.5)$) {$g$};

\end{tikzpicture}
     \caption{The \textit{sailboat} graph $F$ with edges $abc$, $acd$,  $ade$, $aef$, $afg$ and $bdg$ is another simple example with $r(F,n)=n^{\Theta(\log n)}$.}
     \label{fig:sailboat}
 \end{figure}
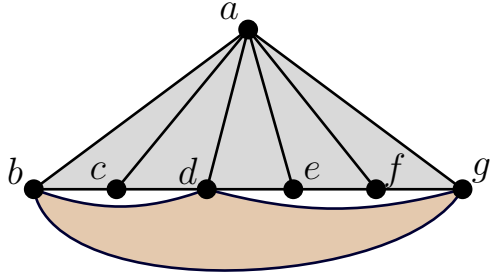

Let $B_n$ be the 3-graph with vertex set $X\cup Y$ where $X$ and $Y$ are disjoint sets of size $n$, and whose edge set consists of all triples that intersect both $X$ and $Y$. Then for a 3-graph $F$, by definition the Ramsey number $r(F,B_{n})$ is the smallest integer $N$ such that every $F$-free 3-graph $H$ of order $N$ contains two disjoint vertex sets $X$ and $Y$, each of size $n$, such that no edge of $H$ intersects both $X$ and $Y$.

\begin{thm}\label{thm:bn}
For $t\ge 2$,
$$
r(H_t,B_{n})\leq n^{O(\log^{2t-4}(n))}.
$$
\end{thm}

Note that $r(H_2,B_{n})= n^{\Theta(1)}$ while $r(H_2,n)= n^{\Theta(\log n)}$. Ramsey numbers involving $B_{n}$ were generally suspected to be indistinguishable in growth rate from Ramsey numbers involving the complete $3$-graph $K_n^{(3)}$ (see the Erd\H{o}s problem~\cite[Problem 11.4]{Mubayi2020survey}), so the contrast here comes as a surprise.

\section{Proof of Theorem \ref{thm:iteratedbicone} and Theorem \ref{thm:bn}}\label{sec:2}
We begin with the definition of $H_t$. For simplicity we will use $abc$ to denote an edge $\{a,b,c\}$ in a 3-graph.

\begin{defn}
Given a $3$-graph $H$ with a partition $X,Y$ of the vertex set $V(H)$, define the bicone of $H$, denoted by $\bicone(H;X,Y)$, to be the $3$-graph with vertex set $\{x,y,z\}\sqcup V(G)$, introducing three new vertices $x, y, z$, and edge set
\[
E(\bicone(H;X,Y))=\{xyz\}\cup\{zxv\mid v\in Y\}\cup\{zyw\mid w\in X\}\cup E(H).
\]
% The vertex $z$ is called the \textit{center} of the bicone.
\end{defn}

\begin{defn} For $t\geq 0$, we define the \textit{iterated bicone} $H_t$ as following: let $H_0=\emptyset$ be the empty graph with partition $X_0=Y_0=\emptyset$. Let $H_1=K_3^{(3)}=\bicone(\emptyset)$ be a single edge with partition $X_1=\{x,z\},Y_1=\{y\}$, and more generally for all $t\geq 1$, we define $H_t=\bicone(H_{t-1};X_{t-1},Y_{t-1})$, with partition $X_t=X_{t-1}\cup \{x,z\},\ Y_t=Y_{t-1}\cup \{y\}$.
\end{defn}

Give a 3-graph $H$ and a vertex $v$, the \textit{link} of $v$ in $H$ consists of pairs $\{a,b\}$ such that $vab$ is an edge in $H$, i.e. 
$$
\{\{a,b\}\subseteq V(H):vab\in E(H)\}.
$$

\begin{figure}
    \centering
\begin{tikzpicture}[scale=0.6,line cap=round,line join=round]

% -------------------- vertices (positions) --------------------
\coordinate (z) at (0,0);
\coordinate (y) at (2.8,2.2);
\coordinate (x) at (2.8,-2.2);

\coordinate (c) at (6.8,-2.2);
\coordinate (b) at (8.8, 2.2);
\coordinate (a) at (10.8,-2.2);

% -------------------- hyperedges (filled regions) --------------
% def
\fill[black!60,opacity=0.25] (x)--(y)--(z)--cycle;

% fdb
\fill[black!60,opacity=0.25] (z)--(x)--(b)--cycle;

% fea
\fill[black!60,opacity=0.25] (z)--(y)--(a)--cycle;

% fec
\fill[black!60,opacity=0.25] (z)--(y)--(c)--cycle;

% abc ("triangle")
\fill[fill=brown!60,opacity=0.70] (b)--(a)--(c)--cycle;

% -------------------- hyperedge boundaries ---------------------
\draw[line width=1.0] (x)--(y)--(z)--cycle;            % def

\draw[line width=1.0] (z)--(x)--(b)--cycle;            % fdb
\draw[line width=1.0] (z)--(y)--(a)--cycle;            % fea
\draw[line width=1.0] (z)--(y)--(c)--cycle;            % fec

\draw[line width=1.0] (b)--(a)--(c)--cycle;;     % abc

% -------------------- vertices (dots) --------------------------
\foreach \p in {x,y,z,a,b,c}
  \fill (\p) circle (7pt);

% -------------------- labels -----------------------------------
\node[scale=1.5] at ($(z)+(-0.50,-0.5)$) {$z$};
\node[scale=1.5] at ($(y)+(-0.5,0.5)$) {$y$};
\node[scale=1.5] at ($(x)+(-0.5,-0.5)$) {$x$};

\node[scale=1.5] at ($(c)+(0.5,-0.5)$) {$c$};
\node[scale=1.5] at ($(b)+(0.5,0.5)$) {$b$};
\node[scale=1.5] at ($(a)+(0.5,-0.5)$) {$a$};

\end{tikzpicture}
    \caption{Another illustration of $H_2$.}
    % \xh{complete sentence}
    \label{fig:H2}
\end{figure}
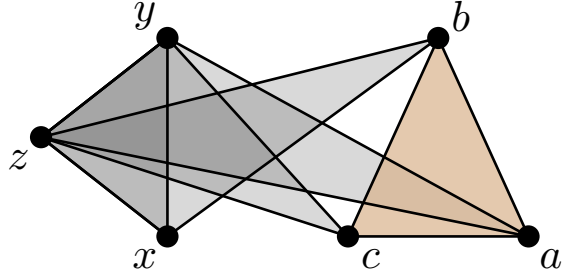

 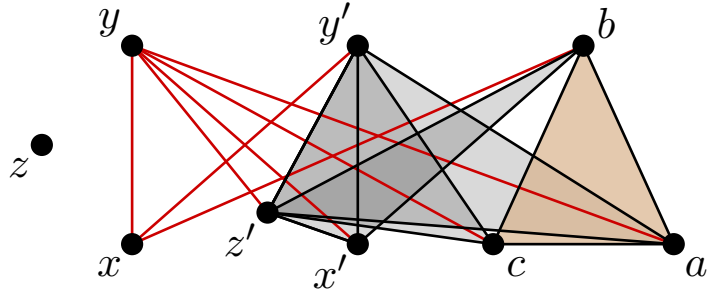
\begin{figure}
    \centering
\begin{tikzpicture}[scale=0.6,line cap=round,line join=round]

% -------------------- vertices (positions) --------------------
\coordinate (z) at (0,0);
\coordinate (y) at (2,2.2);
\coordinate (x) at (2,-2.2);

\coordinate (z') at (5,-1.5);
\coordinate (y') at (7,2.2);
\coordinate (x') at (7,-2.2);

\coordinate (c) at (10,-2.2);
\coordinate (b) at (12, 2.2);
\coordinate (a) at (14,-2.2);

% -------------------- hyperedges (filled regions) --------------

% z'x'b
\fill[black!60,opacity=0.25] (z')--(x')--(b)--cycle;

% z'y'a
\fill[black!60,opacity=0.25] (z')--(y')--(a)--cycle;

% z'y'c
\fill[black!60,opacity=0.25] (z')--(y')--(c)--cycle;

% abc ("triangle")
\fill[fill=brown!60,opacity=0.70]
  (b)--(a)--(c)--cycle;

% -------------------- links of z ---------------------
\draw[red!80!black,line width=1.0] (x)--(y);            % xyz

\draw[red!80!black,line width=1.0] (x)--(b);            % zxb
\draw[red!80!black,line width=1.0] (y)--(a);            % zya
\draw[red!80!black,line width=1.0] (y)--(c);            % zyc
\draw[red!80!black,line width=1.0] (x)--(y');            % zxy'
\draw[red!80!black,line width=1.0] (y)--(x');            % zyx'
\draw[red!80!black,line width=1.0] (y)--(z');            % zyz'

% -------------------- hyperedge boundaries ---------------------

\draw[line width=1.0] (x')--(y')--(z')--cycle;            % x'y'z'

\draw[line width=1.0] (z')--(x')--(b)--cycle;            % z'x'b
\draw[line width=1.0] (z')--(y')--(a)--cycle;            % z'y'a
\draw[line width=1.0] (z')--(y')--(c)--cycle;            % z'y'c

\draw[line width=1.0] (b)--(a)--(c)--cycle;;     % abc

% -------------------- vertices (dots) --------------------------
\foreach \p in {x,y,z,a,b,c,x',y',z'}
  \fill (\p) circle (7pt);

% -------------------- labels -----------------------------------
\node[scale=1.5] at ($(z)+(-0.50,-0.5)$) {$z$};
\node[scale=1.5] at ($(y)+(-0.5,0.5)$) {$y$};
\node[scale=1.5] at ($(x)+(-0.5,-0.5)$) {$x$};

\node[scale=1.5] at ($(x')+(-0.6,-0.6)$) {$x'$};
\node[scale=1.5] at ($(y')+(-0.5,0.5)$) {$y'$};
\node[scale=1.5] at ($(z')+(-0.6,-0.6)$) {$z'$};

\node[scale=1.5] at ($(c)+(0.5,-0.5)$) {$c$};
\node[scale=1.5] at ($(b)+(0.5,0.5)$) {$b$};
\node[scale=1.5] at ($(a)+(0.5,-0.5)$) {$a$};

\end{tikzpicture}
    \caption{An illustration of $H_3$. The red segments denote the link of $z$.}
    % \xh{complete sentence}
    \label{fig:iterbicone3}
\end{figure}

% \begin{figure}[h]
%     \centering
%     \includegraphics[scale=0.3]{Figures/iterbicone3.png}
% \caption{The iterated bicone $H_3$, a 3-graph with 12 edges. The green(respectively, blue, red) vertex is the center of $H_3$(respectively, $H_2$, $H_1$) and the green(blue, red) segments denote the link of the center in $H_3$($H_2$, $H_1$).}
% \label{fig:iterbicone3}
% \end{figure}

The following lemma shows that if a 3-graph $G$ has a dense cut, then there is a pair of large vertex sets which can be used to embed a bicone. For the following proof, we say a $3$-graph $G$ is bipartite with parts $U$ and $V$ if every edge of $G$ intersects both $U$ and $V$.

\begin{lem}\label{lem:densitybicone} Let $G$ be a bipartite $3$-graph with parts $V_0$ and $V_1$ each of size at most $N$, and suppose $|E(G)|\geq \eps N^3$. Then there exist vertices $x\in V_0$, $y\in V_1$, and $z\in V(G)$ and sets $U_0\subseteq V_0\setminus\{x,z\}$ and $U_1\subseteq V_1\setminus \{y,z\}$ each of size at least $\eps N/2-1$ such that 
\begin{equation}\label{eq:biconeframework}
\{xyz\}\cup \{zxv\mid v\in U_1\}\cup \{zyw\mid w\in U_0\}\subseteq E(G).
\end{equation}
\end{lem}

\begin{proof} Without loss of generality, suppose there are at least $\eps N^3/2$ edges of $G$ have two ends in $V_0$ and one in $V_1$; consider the subgraph $G'$ consisting of these ``2--1 edges". By averaging, there exists a vertex $z\in V_0$ contained in at least $\eps N^2$ edges of $G'$. Let $L$ be the link of $z$ in $G'$, i.e.
$$
L=\{\{x,y\}:xyz\in E(G')\}.
$$
So $L$ is a bipartite 2-graph with parts $V_0$,$V_1$. 

Starting from $L$, iteratively delete vertices of degree less than $\varepsilon N/2$. This process yields a nonempty subgraph $L'$ with minimum degree at least $\varepsilon N/2$. Choose any edge $xy \in E(L')$, where $x \in V_0$ and $y \in V_1$, and define
\[
U_0 := N_{L'}(y)\setminus\{x\}, \qquad
U_1 := N_{L'}(x)\setminus\{y\}.
\]
By the minimum degree condition, we have $|U_0|, |U_1| \ge \varepsilon N/2 - 1$. By the definition of the link graph $L$, all triples listed in \eqref{eq:biconeframework} are edges of $G$, completing the proof.
\end{proof}
% \lp{can we combine some of the paragraphs above (just removing the extra space) ? It feels like a lot}

The following parameter, which tracks the size of the largest iterated bicone, plays a key role in our inductive proof.

\begin{defn} 
For a $3$-graph $G$ and sets $U_0, U_1 \subseteq V(G)$, define
\[
\omega_{\bcn}(G, U_0, U_1)
\]
to be the maximum $t$ such that $G$ contains a copy of $H_t$ in which $X_t$ maps into $U_0$ and $Y_t$ maps into $U_1$. We say that such a copy of $H_t$ is {ordered from $U_0$ to $U_1$}.

More generally, we say that $G$ contains a copy of $H_t$ {across the cut} $(U_0, U_1)$ if the copy is ordered either from $U_0$ to $U_1$ or from $U_1$ to $U_0$.
\end{defn}

Note that we always have $\omega_{\bcn}(G,U_0,U_1)\geq 0$ since the empty graph is a trivial subgraph, and $\omega_{\bcn}(G,U_0,U_1)\geq 1$ exactly when there is some $2$--$1$ edge in the cut $(U_0,U_1)$ of $G$.
The following is our key lemma.

\begin{lem}\label{lem:iteratedbiconeinduction}
For all integers $t\ge 2$ and $1\le s\le 2t-2$, there exists a constant $c>0$ such that for every $n \ge 2$ the following holds.  
If $G$ is a $3$-graph on
\[
N \ge 2^{c \log^{s} n}
\]
vertices, then $G$ contains either a copy of $H_t$, an independent set of size $n$, or disjoint vertex sets $U_0, U_1 \subseteq V(G)$ with
\[
|U_0|, |U_1| \ge 2^{-c \log^{s} n} \, N
\]
such that
\[
\omega_{\bcn}(G, U_0, U_1) + \omega_{\bcn}(G, U_1, U_0) \le 2t - 3 - s.
\]
\end{lem}

\begin{proof}
By adjusting the constant $c$ if necessary, we may assume that $n$ is a power of $2$, so that $\log n$ is an integer with $\log n \ge 1$. We proceed by induction on $s$. For the base case, let $s=1$ and take $c=10$. Let $G$ be a $3$-graph on $N \ge 2^{c \log n}$
vertices that contains neither a copy of $H_t$ nor an independent set of size $n$.

We first consider the case $|E(G)| \le N^3/(9n^2)$. Sample each vertex of $G$ independently with probability $p = 2n/N$, and then delete one vertex from each sampled edge. The resulting set is independent in $G$, and its expected size is
\[
pN - p^3 |E(G)|
\;\ge\;
pN - \frac{p^3 N^3}{9n^2}
\;>\;
n.
\]
This contradicts the assumption that $G$ contains no independent set of size $n$.

We may therefore assume that $|E(G)| \ge N^3/(9n^2)$. Partition $V(G)$ uniformly at random into two parts $V_0$ and $V_1$. Each edge of $G$ intersects both parts with probability $3/4$, so there exists a choice of $(V_0,V_1)$ such that the induced bipartite subhypergraph $G[V_0,V_1]$ contains at least $N^3/(12n^2)$ edges.

Applying Lemma~\ref{lem:densitybicone} to $G[V_0,V_1]$, we obtain vertices $x,y,z \in V(G)$ and sets $U_0, U_1 \subseteq V(G)$ satisfying \eqref{eq:biconeframework}, with
\[
|U_0|, |U_1| \ge \frac{N}{24n^2} - 1 \ge 2^{-c \log n} N.
\]

Finally, we claim that
\[
\omega_{\bcn}(G,U_0,U_1), \ \omega_{\bcn}(G,U_1,U_0) \le t-2.
\]
Indeed, if $G$ contained a copy of $H_{t-1}$ across the cut $(U_0,U_1)$, then adding the vertices $x,y,z$ would extend this to a copy of $H_t$, contradicting our assumption. This completes the proof of the base case.

Suppose the lemma holds for some $1\le s<2t-2$ with constant $c$ and we will prove it for $s+1$ with constant $C=10c$. Suppose $G$ is a $3$-graph on $N\geq 2^{C\log^{s+1}(n)}$ vertices not containing $H_t$ or an independent set of size $n$. By the inductive hypothesis, $G$ contains disjoint $V_0,V_1\subseteq V(G)$ of size at least $2^{-c\log^s(n)}N$ satisfying the lemma for $(s,t)$. Iterate inside each part $\log n$ times according to the following recursion: for each binary word $w\in \{0,1\}^{<\log n}$, we have some $V_w$ with
\[
|V_w|\geq 2^{-|w|c\log^s(n)}N\geq 2^{(C\log n-c|w|)\log^s(n)}\geq 2^{c\log^s(n)},
\]
so by inductive hypothesis there exist disjoint $V_{w0},V_{w1}\subseteq V_w$ of size at least 
$$
2^{-c\log^s(n)}|V_w|\geq 2^{-(|w|+1)c\log^s(n)}N
$$ 
such that 
\begin{equation}\label{eq:omega}
\omega_{\bcn}(G,V_{w0},V_{w1})+\omega_{\bcn}(G,V_{w1},V_{w0})\leq 2t-3-s.    
\end{equation}
This generates a family $\{V_w\mid w\in \{0,1\}^{\leq \log n}\}$ of nested subsets of $V(G)$. We split into two cases.

\medskip

\noindent\textbf{Case 1:}~There exists $w\in \{0,1\}^{<\log n}$ such that $|E(G[V_{w0},V_{w1}])|\geq 4\cdot 2^{-C\log^{s+1}(n)}N^3$.

Note that this case is impossible if $s=2t-3$, since by (\ref{eq:omega}) we have $\omega_{\bcn}(G,V_{w0},V_{w1})=\omega_{\bcn}(G,V_{w0},V_{w1})=0$, which implies that $G[V_{w0},V_{w1}]$ is empty. 

In this case, by Lemma~\ref{lem:densitybicone} applied to the bipartite subgraph, there exists $x,y,z$ and $U_0\subseteq V_{w0}$, $U_1\subseteq V_{w1}$ of size at least $2^{2-C\log^{s+1}(n)}N/2-1\geq 2^{-C\log^{s+1}(n)}N$ satisfying \eqref{eq:biconeframework}. Without loss of generality suppose $z\in V_{w0}$, and then by \eqref{eq:biconeframework} we claim
\[
\omega_{\bcn}(G,V_{w0},V_{w1}) \geq \omega_{\bcn}(G,U_0,U_1)+1.
\]
This is because for any copy of $H_{t'}$ ordered from $U_0$ to $U_1$, we may extend this to $H_{t'+1}$ ordered from $V_{w0}$ to $V_{w1}$ by adding the vertices $x,z\in V_{w0}$ and $y\in V_{w1}$. Then by choice of $V_{w0},V_{w1}$, we have 
\[\omega_{\bcn}(G,U_0,U_1)+\omega_{\bcn}(G,U_1,U_0)\leq 2t-3-s-1,\]
as desired. 

\medskip

\noindent\textbf{Case 2:}~For all $w\in \{0,1\}^{<\log n}$, we have $|E(G[V_{w0},V_{w1}])|\leq 4\cdot2^{-C\log^{s+1}(n)}N^3$.

In this case, consider the family $\cS=\{V_w\mid w\in \{0,1\}^{\log n}\}$ containing $n$ sets of size at least $2^{-c\log^{s+1}(n)}N$. Call an edge $xyz\in E(G)$ a \textit{crossing edge} of $\cS$ if $x,y,z$ lie in distinct members of $\cS$. Observe that every crossing edge $xyz$ must have some unique word $w\in \{0,1\}^{<\log n}$ such that $xyz\in E(G[V_{w0},V_{w1}])$; this is the longest common prefix $w$ such that $x,y,z\in V_w$. There are $\sum_{k<\log n}2^k=n-1$ such words, so the number of crossing edges is at most 
$$
4n\cdot 2^{-C\log^{s+1}(n)}N^3.
$$ 
Now, form a set $S\subseteq V(G)$ of size $n$ by choosing one vertex from each $V_w\in \cS$ uniformly at random. Each crossing edge is contained in $S$ with probability at most $(2^{-c\log^{s+1}(n)}N)^{-3}$, so by Markov's inequality,
\[
\mathbb P[S\text{ is not independent in }G]\le 4n\cdot 2^{-C\log^{s+1}(n)}N^3\cdot (2^{-c\log^{s+1}(n)}N)^{-3}<1,
\]
having set $C=10c\ge 3c+1$. Thus $G$ contains an independent set of size $n$, a contradiction.
\end{proof}

Now we are ready to prove Theorem~\ref{thm:bn} and Theorem~\ref{thm:iteratedbicone}.

\begin{proof}[Proof of Theorem~\ref{thm:bn}]
Apply Lemma~\ref{lem:iteratedbiconeinduction} with $s = 2t-3$ and $n = 2m$. Note that $s \ge 1$ for all $t \ge 2$. Then any $3$-graph $G$ on $m \cdot 2^{c \log^{2t-3}(2m)}$ vertices contains either a copy of $H_t$, an independent set of size $2m$, or disjoint vertex sets $U_0, U_1$ each of size at least $m$ such that $G$ contains no copy of $H_1$ across the cut $(U_0,U_1)$.

Observe that the second outcome implies the third. Consequently, we obtain
\[
r(H_t, B_m) \le m \cdot 2^{c \log^{2t-3}(2m)}
            \le m^{O(\log^{2t-4} m)}.
\qedhere\]
\end{proof}

\begin{proof}[Proof of Theorem~\ref{thm:iteratedbicone}]
Apply Lemma~\ref{lem:iteratedbiconeinduction} with $s = 2t-2$. Let $G$ be a $3$-graph on $2^{c \log^{2t-2} n}$ vertices. Recall that $\omega_{\bcn}(G,V_0,V_1) \ge 0$ for all nonempty vertex sets $V_0, V_1$. Since $2t - 3 - s < 0$, the third outcome in Lemma~\ref{lem:iteratedbiconeinduction} cannot occur. Consequently, $G$ contains either a copy of $H_t$ or an independent set of size $n$. This implies that
\[
r(H_t,n) \le 2^{c \log^{2t-2} n}\le n^{O(\log^{2t-3} n)},
\]
as claimed.
\end{proof}

\section{Concluding Remarks}
 For simplicity, we have focused on the family $\{H_t\}_{t\ge 1}$. Using standard machinery, it is possible to find a much larger family of $3$-graphs satisfying \Cref{thm:iteratedbicone}. For example,~\cite[Proposition 6.2]{fox2021independent} implies the following.
 \begin{thm}\label{thm:blowup}
 For $t\ge 2$, if $F$ is a subgraph of an iterated blowup of $H_t$, then 
 $$
 r(F,n)\le n^{O(\log^{2t-3} (n))}.
 $$
 \end{thm}
 In particular, if $F$ is a subgraph of an iterated blowup of $H_2$ that is not iterated tripartite, then $r(F,n)=n^{\Theta(\log n)}$; see Figure~\ref{fig:sailboat} for a simple example.

It would be extremely interesting to determine all $F$ obtaining quasipolynomial growth rate, but such a general classification seems out of reach of our methods. As a first step, one may instead consider the following simpler problem.
\begin{prob}
Classify hypergraphs $F$ with exactly two tightly connected components according to the growth rate of $r(F,n)$.
\end{prob}

When one of the two components consists of a single edge, our results, together with earlier work, yield the following classification.

\begin{thm}
Let $F$ be a $3$-graph with two tightly connected components $F_1$ and $F_2$ where $F_2$ is a single edge. Then
\[
r(F,n)=\begin{cases}
n^{\Theta(1)}&\text{ if $F$ is iterated tripartite,}\\
n^{\Theta(\log n)}&\text{ if $F$ is not iterated tripartite but $F_1$ is,}\\
2^{n^{\Theta(1)}}&\text{ if $F_1$ is not iterated tripartite.}\\
\end{cases}
\]
\end{thm}

\begin{proof}
If $F$ is iterated tripartite, then the work of Erd\H{o}s and Hajnal~\cite{erdos1972ramsey} implies
$r(F,n)=n^{\Theta(1)}$.
If $F_1$ is not iterated tripartite then the work of Conlon, Fox, Gunby, He, Mubayi, Suk, Verstra\"ete and Yu~\cite{conlon2024when}
implies
\[
r(F,n)\ge r(F_1,n)=2^{n^{\Theta(1)}}.
\]
Finally, if $F$ is not iterated tripartite then note that $F$ has two tight components, so by
Theorem~\ref{thm:lower_bound} of Conlon, Fox, Gunby, He, Mubayi, Suk, Verstra\"ete and Yu, we have
\[
r(F,n)\ge n^{\Omega(\log n)}.
\]
Now it suffices to show that in this case, $F$ is contained in a blowup of $H_2$ with vertex set
$\{x,y,z,a,b,c\}$ (see Figure~\ref{fig:H2}).
Since $F_1$ is tripartite, there is a proper tripartition $(X,Y,Z)$ of $V(F)$.
Let $e$ be the single edge of $H_2$.
Since $F$ is not iterated tripartite, without loss of generality we may assume
$|e\cap X|=2$ and $|e\cap Y|=1$.
Then the map given by
\[
e\cap X\mapsto \{b,c\},\quad
e\cap Y\mapsto a,\quad
X\setminus e\mapsto x,\quad
Y\setminus e\mapsto y,\quad
Z\mapsto z
\]
is a homomorphism from $F$ to $H_2$, so by Theorem~\ref{thm:blowup} we have
\[
r(F,K_3^{(3)})\le n^{O(\log n)}.
\qedhere\]
\end{proof}

So far, we have seen 3-graphs $F$ for which $r(F,n)$ is of order $n^{\Theta(1)}$ (iterated tripartite 3-graphs~\cite{erdos1972ramsey}), of order $n^{\log^{\Theta(1)}(n)}$ (subgraphs of iterated bicones that are not iterated tripartite, as studied in the current paper), or of order $n^{\Theta(n)}$ (link hypergraph of non-bipartite graphs~\cite{fox2021independent}). However, to the best of our knowledge, no $3$-graph $F$ is known whose Ramsey number exhibits a growth rate strictly between these regimes. This motivates the following natural and compelling problem.

\begin{prob}
Does there exist a 3-graph $F$ such that $r(F,n)$ is neither $n^{O(1)}$, $n^{\log^{\Theta(1)}(n)}$, nor $n^{\Omega(n)}$?   
\end{prob}

\noindent \textbf{Acknowledgments.} The authors are grateful to Jacob Fox and Hans Hung-Hsun Yu for stimulating conversations.


\begin{thebibliography}{99}
\bibitem{AKS1980}
M. Ajtai, J. Komlós, and E. Szemerédi, A note on Ramsey numbers, {\it J. Combin. Theory Ser.~A} {\bf 29} (1980), 354-360.

\bibitem{AlonSpencer}
N. Alon and J. Spencer, {\bf The Probabilistic Method}, Fourth Edition, Wiley Publishing (2016).

\bibitem{ascoli2025polynomial}
R. Ascoli, X. He, and H. Yu, Polynomial-to-exponential transition in 3-uniform Ramsey numbers (2025), preprint available at arXiv:2507.09434.

\bibitem{bohman2010early}
T. Bohman and P. Keevash, The early evolution of the {$H$}-free process, {\it Invent. Math.} {\bf 181} (2010), 291--336.

\bibitem{caro2000asymptotic}
Y. Caro, Y. Li, C. Rousseau, and Y. Zhang, Asymptotic bounds for some bipartite graph: complete graph Ramsey numbers, {\it Discrete Math.} {\bf 220} (2000), 51-56.

\bibitem{chvatal1977tree}
V. Chv\'atal, Tree-complete graph {R}amsey numbers, {\it J. Graph Theory} {\bf 1} (1977), 93.

\bibitem{conlon2024ramseyzarankiewicz}
D. Conlon, S. Mattheus, D. Mubayi, and J. Verstra\"ete, Ramsey numbers and the {Z}arankiewicz problem, {\it Bull. Lond. Math. Soc.} {\bf 56} (2024), 2014--2023.

\bibitem{conlon2010hypergraph}
D. Conlon, J. Fox, and B. Sudakov, Hypergraph {R}amsey numbers, {\it J. Amer. Math. Soc.} {\bf 23} (2010), 247--266.

\bibitem{conlon2024when}
D. Conlon, J. Fox, B. Gunby, X. He, D. Mubayi, A. Suk, J. Verstra\"ete, and H. Yu, When are off-diagonal hypergraph {R}amsey numbers polynomial?, {\it Proc. Amer. Math. Soc.} {\bf 153} (2025), 4605--4617.

\bibitem{erdos1972ramsey}
P. Erd\H{o}s and A. Hajnal, On {R}amsey like theorems. {P}roblems and results, in {\it Combinatorics ({P}roc. {C}onf. {C}ombinatorial {M}ath., {M}ath. {I}nst., {O}xford, 1972)}, Inst. Math. Appl., Southend-on-Sea (1972), 123--140.

\bibitem{erdos1952combinatorial}
P. Erd\H{o}s and R. Rado, Combinatorial theorems on classifications of subsets of a given set, {\it Proc. London Math. Soc. (3)} {\bf 2} (1952), 417--439.

\bibitem{fox2021independent}
J. Fox and X. He, Independent sets in hypergraphs with a forbidden link, {\it Proc. Lond. Math. Soc. (3)} {\bf 123} (2021), 384--409.

\bibitem{FoxSudakovDRC}
J. Fox and B. Sudakov, Dependent random choice, {\it Random Struct. Algorithms} {\bf 38} (2009).

\bibitem{hefty2025improving}
Z. Hefty, P. Horn, D. King, and F. Pfender, Improving $ R (3, k) $ in just two bites (2025), preprint available at arXiv:2510.19718.

\bibitem{Kim1995}
J. Kim, The Ramsey number $R(3, t)$ has order of magnitude $t^2/\log t$, {\it Random Struct. Algorithms} {\bf 7} (1995), 173-207.

\bibitem{mattheus2023asymptotics}
S. Mattheus and J. Verstraete, The asymptotics of {$r(4,t)$}, {\it Ann. of Math. (2)} {\bf 199} (2024), 919--941.

\bibitem{MV2019}
D. Mubayi and J. Verstra\"ete, A note on pseudorandom {R}amsey graphs, {\it J. Eur. Math. Soc. (JEMS)} {\bf 26} (2024), 153--161.

\bibitem{MubayiRazborov2021}
D. Mubayi and A. Razborov, Polynomial to exponential transition in {R}amsey theory, {\it Proc. Lond. Math. Soc.} {\bf 122} (2021), 69--92.

\bibitem{Mubayi2020survey}
D. Mubayi and A. Suk, A Survey of Hypergraph Ramsey Problems, in {\bf Discrete Mathematics and Applications}, Springer International Publishing (2020), 405--428.

\bibitem{SHE1983}
J. Shearer, A note on the independence number of triangle-free graphs, {\it Discrete Math.} {\bf 46} (1983), 83-87.

\end{thebibliography}
\end{document}